\documentclass[reqno, 12pt]{amsart}
\usepackage{amssymb}
\usepackage{enumerate}
\usepackage[color, notref, notcite]{showkeys}     
\definecolor{refkey}{gray}{.5}   
\definecolor{labelkey}{gray}{.5} 
\usepackage{graphicx}

\theoremstyle{plain}
\newtheorem{theorem}{Theorem}[section]
\newtheorem{lemma}[theorem]{Lemma}
\newtheorem{proposition}[theorem]{Proposition}
\newtheorem{corollary}[theorem]{Corollary}
\theoremstyle{definition}
\newtheorem{definition}[theorem]{Definition}
\newtheorem{remark}[theorem]{Remark}

\numberwithin{equation}{section}

\makeatletter
\def\@map#1#2[#3]{\mbox{$#1 \colon #2 \longrightarrow #3$}}
\def\map#1#2{\@ifnextchar [{\@map{#1}{#2}}{\@map{#1}{#2}[#2]}}
\makeatother
\newcommand{\set}[2]{\ensuremath{\left\{#1 \,\colon #2\right\}}}
\newcommand{\abs}[1]{\ensuremath{\left| #1 \right|}}
\newcommand{\R}{\mathbb{R}}

\newcommand{\C}{\mathbb{C}}

\newcommand{\eps}{\varepsilon}
\newcommand{\te}{\vartheta}
\renewcommand{\phi}{\varphi}

\newcommand{\uom}{\underline{\omega}}
\newcommand{\om}{\Omega}
\newcommand{\pnm}{\phi_{n,m}}
\newcommand{\id}{\operatorname{id}}
\title{Random interval homeomorphisms}
\author{Llu\'{\i}s Alsed\`a}
\address{Departament de Matem\`atiques, Edifici Cc, Universitat
Aut\`onoma de Barcelona, 08913 Cerdanyola del Vall\`es, Barcelona,
Spain}
\email{alseda@mat.uab.cat}
\author{Micha\l\ Misiurewicz}
\address{Department of Mathematical Sciences, IUPUI, 402 N. Blackford
    Street, Indianapolis, IN 46202}
\email{mmisiure@math.iupui.edu}
\thanks{The first author has been partially supported by the MEC grant
numbers MTM2008-01486 and MTM2011-26995-C02-01.}

\date{December 12, 2012}

\begin{document}
\begin{abstract}
We investigate homeomorphisms of a compact interval, applied randomly.
We consider this system as a skew product with the two-sided Bernoulli
shift in the base. If on the open interval there is a metric in which
almost all maps are contractions, then (with mild additional
assumptions) there exists a global pullback attractor, which is a
graph of a function from the base to the fiber. It is also a forward
attractor. However, the value of this function depends only on the
past, so when we take the one-sided shift in the base, it disappears.
We illustrate those phenomena on an example, where there are two
piecewise linear homeomorphisms, one moving points to the
right and the other one to the left.
\end{abstract}
\maketitle

\section{Introduction}

In this paper we investigate the properties of the systems of randomly
applied orientation preserving homeomorphisms of the compact interval
$[0,1]$. Such a system can be considered as a skew product with a
mixed topological-measure structure. In the base we do not need any
topology (although sometimes we have it), but we assume that we have
there an ergodic measure preserving transformation of a probability
space. In the fiber, which is an interval, we have orientation
preserving homeomorphisms, depending in a measurable way on the point
in the base.

We are interested in the existence of almost global attractors which
are graphs of measurable functions from the base to the fiber. When we
speak of an attractor, we mean a set towards which almost all orbits
converge, and the convergence is considered fiberwise (only in the
direction of a fiber). This agrees with the philosophy saying that the
phase space is really only the fiber space (here, the interval).

Those systems and their attractors can be looked upon from various
points of view (random systems, Strange Nonchaotic Attractors,
Iterated Function Systems, nonautonomous systems, etc.),
see~\cite{AM}.

Our main result is a detailed description of the behavior of a certain
one-parameter family of piecewise linear random homeomorphisms.
However, we precede it with some general results, which can be applied
to very general random systems of interval homeomorphisms.

Note that 0 and 1 are fixed points of all orientation preserving
homeomorphisms of $[0,1]$, so the products of the base space with
$\{0\}$ and with $\{1\}$ (we will refer to those sets as \emph{level
  $0$} and \emph{level $1$}) are invariant for the skew product. if
they are attracting in the sense of negative fiberwise Lyapunov
exponent, one expects their basins of attraction to have positive
measure. We prove that this is the case in a general situation, under
some mild additional conditions. Our proof uses the same ideas as the
proof by Bonifant and Milnor~\cite{BoMi} in the special case.

In~\cite{BoMi} the cases of attracting levels 0 and 1 (when the fiber
maps have negative Schwarzian derivative) and repelling levels 0 and 1
(when the fiber maps have positive Schwarzian derivative) were
considered separately. Here we join them together and consider an
invertible map in the base. The two opposite types of behavior are
observed depending whether the time goes to $+\infty$ or to $-\infty$.
The common boundary of the basins of attraction of the levels 0 and 1
as the time goes to $-\infty$ is a graph of a measurable function from
the base to the interval, is a forward attractor
(statement~\eqref{ma3} of Theorem~\ref{main}) and a pullback attractor
(statement~\eqref{ma4} of Theorem~\ref{main}) for the system.

In this general theorem one needs an additional assumption that the
maps in the fibers are kind of contractions almost everywhere. Proving
it is crucial in the study of this problem. In~\cite{BoMi} this is achieved by the
assumptions on Schwarzian derivatives of the maps. In our
one-parameter family of maps this requires a careful proof. In fact,
the contraction we get is very weak (although really it may turn out
to be exponential almost everywhere; this is unknown to us).

Finally, we compare the invertible and non-invertible cases. Although
the attractor in the invertible case depends only on the past in the
base, it vanishes when we forget about the past (more precisely, it
becomes the whole space). We call it \emph{the mystery of the
  vanishing attractor}. While we described it already in~\cite{AM},
the system considered here is a much better illustration of this
paradox.

The paper is organized as follows. In Section~\ref{sec-basin} we
generalize the theorem of Bonifant and Milnor. In Section~\ref{time}
we consider a general system with an invertible map in the base and
prove a general theorem about its properties. In
Section~\ref{bernoulli} we prove additional properties of the skew
product under the assumption that the system in the base is a
Bernoulli shift. In Section~\ref{plh} we define our family of
piecewise linear homeomorphisms and prove its contraction properties.
In Section~\ref{sec-measures} we investigate our family of systems
from the point of view of invariant measures. In Section~\ref{2s1s} we
compare the invertible and noninvertible systems.

Let us conclude this section with an observation and some questions.
In the theory of interval maps (not random) negative Schwarzian
derivative often substitutes expansion (see, e.g.,~\cite{devan, Mis}).
The same happens in~\cite{BoMi}, where positive Schwarzian derivative
gives us a form of contraction. However, in our piecewise linear
system we also get a kind of contraction. What is the source of it?
Does it have anything to do with some property resembling negative
Schwarzian derivative? Can it be observed in non-random, say unimodal,
maps?

\section{Boundaries of basins of attraction}\label{sec-basin}

Let us start with a very general situation. Let $\om$ be some space
(later there will be an invariant measure on it), $R:\om\to\om$ a map,
$I=[0,1]$, $G:\om\times I \to \om\times I$ a skew product:
$G(\te,x)=(R(\te),g_\te(x))$, and let $\pi_2$ be the projection from
$\om\times I$ to $I$. We assume that each $g_\te$ is an orientation
preserving homeomorphism of $I$ onto itself.

The question is: if the level 0 set $\om\times\{0\}$ is an attractor,
what can we say about the boundary of the basin of attraction?
It can be defined as follows.

Let $\pnm(\te)$ be the unique number such that
\[
G^n(\te,\pnm(\te))=\left(R^n(\te),\frac1m\right).
\]
This defines the function $\pnm:\om\to I$.

\begin{remark}\label{profinm}
Clearly, $\inf_{n\ge N}\pnm(\te)$ is increasing in $N$ and decreasing
in $m$.
\end{remark}

Then we define a function $\phi:\om\to I$ by
\begin{equation}\label{equ0}
\phi(\te)=\lim_{m\to\infty}\lim_{N\to\infty}\inf_{n\ge N}\pnm(\te).
\end{equation}
By Remark~\ref{profinm} the limits above exist.

Now we study the map $\varphi$ defined above.

\begin{lemma}\label{lem1}
If $x<\phi(\te)$ then
\begin{equation}\label{equ1}
\lim_{n\to\infty}\pi_2(G^n(\te,x))=0.
\end{equation}
If $x>\phi(\te)$ then~\eqref{equ1} does not hold.
\end{lemma}

\begin{proof}
Assume first that $x<\phi(\te)$. Then, by Remark~\ref{profinm},
\[
\forall_m\ \ x<\lim_{N\to\infty}\inf_{n\ge N}\pnm(\te),
\]
so
\[
\forall_m\ \exists_N\ \ x<\inf_{n\ge N}\pnm(\te),
\]
so
\[
\forall_m\ \exists_N\ \forall_{n\ge N}\ \ x<\pnm(\te).
\]
Observe that for every $m$ the inequality $x<\pnm(\te)$ is equivalent
to
\[
\pi_2(G^n(\te,x)))<\frac1m,
\]
and~\eqref{equ1} follows.

Assume now that $x>\phi(\te)$. Then, again by Remark~\ref{profinm},
\[
\exists_m\ \ x>\lim_{N\to\infty}\inf_{n\ge N}\pnm(\te),
\]
so
\[
\exists_m\ \forall_N\ \ x>\inf_{n\ge N}\pnm(\te),
\]
so
\[
\exists_m\ \forall_N\ \exists_{n\ge N}\ \ x>\pnm(\te).
\]
Similarly as above, for every $m$ the inequality $x>\pnm(\te)$ is
equivalent to
\[
\pi_2(G^n(\te,x)))>\frac1m,
\]
so~\eqref{equ1} does not hold.
\end{proof}

\begin{remark}\label{invgr}
It is easy to see that if~\eqref{equ1} holds when $x<\phi(\te)$
and does not hold when $x>\phi(\te)$,
then the graph of $\phi$ is $G$-invariant.
\end{remark}

\begin{lemma}\label{lem2}
For a given $\te\in\om$ assume that there exists $\eta>0$ and
$\lambda_n$ ($n=0,1,2,\dots$) such that
\[
g_{R^n(\te)}(x)\le\lambda_n x
\]
for every $n$ and $x \in (0,\eta),$ and
\[
\limsup_{n\to\infty}\frac1n\sum_{k=0}^{n-1}\log\lambda_k<0.
\]
Then $\phi(\te)>0$.
\end{lemma}

\begin{proof}
Take $r$ such that
\[
-r\in\left(\limsup_{n\to\infty}\frac1n
\sum_{k=0}^{n-1}\log\lambda_k, 0\right).
\]
Then for sufficiently large $n$ we have
\[
\prod_{k=0}^{n-1}\lambda_k=\exp\left(n\cdot\frac1n
\sum_{k=0}^{n-1}\log\lambda_k\right)<e^{-nr}.
\]
Thus, we get
\begin{equation}\label{equ2}
\lim_{n\to\infty}\prod_{k=0}^{n-1}\lambda_k=0,
\end{equation}
so, in particular,
\[
\max_n\left\{\prod_{k=0}^{n-1}\lambda_k\right\}<\infty.
\]

Take any
\[
x_0\in\left(0,\frac\eta{\max\left\{1,\max_n\left\{\prod_{k=0}^{n-1}
\lambda_k\right\}\right\}}\right).
\]
Then we get for all $n$
\[
\pi_2(G^n(\te,x_0))\le\prod_{k=0}^{n-1}\lambda_k\cdot x_0<\eta,
\]
and by~\eqref{equ2} we get~\eqref{equ1} with $x$ replaced by $x_0$.
By Lemma~\ref{lem1} we get $\phi(\te)\ge x_0>0$.
\end{proof}

Let us now assume additionally that $\om$ is equipped with an
$R$-invariant ergodic probability measure $\mu$, the maps $g_\te$
depend on $\te$ in a measurable way and they are all differentiable at
0. Let $\Lambda$ be the \emph{exponent at level 0}, that is,
\[
\Lambda=\int_\om g_\te'(0)\;d\mu(\te).
\]
By the Birkhoff Ergodic Theorem, for almost every $\te$ we have
\begin{equation}\label{equ3}
\lim_{n\to\infty}\frac1n\sum_{k=0}^{n-1}\log g_{R^k(\te)}'(0)=
\Lambda.
\end{equation}

\begin{theorem}\label{basin}
Assume that $\Lambda<0$ and that at least one of the following
assumptions is satisfied:
\begin{enumerate}[(i)]
\item\label{i1} the set $\{g_\te:\te\in\om\}$ is finite,
\item\label{i2} all functions $g_\te$ are concave,
\item\label{i3} all functions $g_\te$ are twice differentiable and
there exists a constant $C$ such that $g_\te''(x)/g_\te'(x)\le C$
for all $\te,x$.
\end{enumerate}
Then there exists a measurable function $\phi:\om\to I$, positive
almost everywhere, such that for every $\te\in\om$~\eqref{equ1} holds
if $x<\phi(\te)$ and does not hold if $x>\phi(\te)$.
\end{theorem}

\begin{proof}
The function $\phi$ is defined by~\eqref{equ0} and it has the desired
properties by Lemmas~\ref{lem1} and~\ref{lem2}, provided the
assumptions of Lemma~\ref{lem2} are satisfied for almost every $\te$.
To show that they are satisfied, it is enough to prove that there
exists $\eta>0$ such that for every $\te$ and $x\in(0,\eta)$
\begin{equation}\label{equ4}
\frac{g_\te(x)}x<e^{-\Lambda/2}g_\te'(0)
\end{equation}
(remember that $e^{-\Lambda/2}>1$). Indeed, then we can take in
Lemma~\ref{lem2}
\[
\lambda_k= e^{-\Lambda/2}f_{R^k(\te)}'(0),
\]
and by~\eqref{equ3} we get for almost every $\te$
\[
\limsup_{n\to\infty}\frac1n\sum_{k=0}^{n-1}\log\lambda_k\le
\lim_{n\to\infty}\frac1n\sum_{k=0}^{n-1}\log g_{R^k(\te)}'(0)
-\frac\Lambda2=\frac\Lambda2<0.
\]

Assume first that~\eqref{i1} is satisfied. If $\{g_\te:\te\in\om\}=
\{h_1,\dots,h_m\}$, then for every $i$, by the definition of the
derivative and since $h_i(0)=0$, there is $\eta_i>0$ such that for all
$x\in(0,\eta_i)$ we have $h_i(x)/x<e^{-\Lambda/2}h_i'(0)$. Now we
take $\eta=\min\{\eta_1,\dots,\eta_m\}$ and then for every
$x\in(0,\eta)$~\eqref{equ4} holds.

Assume now that~\eqref{i2} is satisfied. Then for every $\te$ and $x$ we
have
\[
\frac{g_\te(x)}x\le g_\te'(0)<e^{-\Lambda/2}g_\te'(0)
\]
and we are done.

Assume finally that~\eqref{i3} is satisfied. Set
\[
\eta=\min\left\{1,-\frac\Lambda{2C}\right\}.
\]
Suppose that there are some $\te$ and $x\in(0,\eta)$ for
which~\eqref{equ4} does not hold. Then, by the Mean Value Theorem,
there is $y\in(0,x)$ such that
$g_\te'(y)\ge e^{-\Lambda/2}g_\te'(0)$, that is,
\[
\log g_\te'(y)-\log g_\te'(0)\ge-\frac\Lambda2.
\]
Then there is $z\in(0,y)$ such that
\[
\frac{g_\te''(z)}{g_\te'(z)}=(\log g_\te')'(z)\ge-\frac\Lambda{2y}
>-\frac\Lambda{2\eta}\ge C,
\]
a contradiction with the assumption~\eqref{i3}. This completes the
proof.
\end{proof}

\section{Two directions of time}\label{time}

Let us consider a skew product similar to the one from the preceding
section, under an additional assumption that the map in the base is
invertible. Then we can investigate what happens when the time goes to
$+\infty$ and what happens when it goes to $-\infty$. To be in
agreement with the theory of Strange Nonchaotic Attractors, we will
think of the phenomena from the preceding section as occurring as the
time goes to $-\infty$. Thus, we need new notation.

As before $\om$ is a space with a probability measure $\mu$. Now,
$S:\om\to\om$ is an invertible measurable map (with $S^{-1}$ also
measurable), for which $\mu$ is invariant and ergodic. The map
$F:\om\times I \to \om\times I$ is a skew product, given by
$F(\te,x)=(S(\te),f_\te(x))$, and each $f_\te$ is an orientation
preserving homeomorphism of $I$ onto itself.

We assume that the maps $f_\te$ are differentiable at 0 and 1, and
define
\[
\Lambda_0=\int_\om f_\te'(0)\;d\mu(\te),\qquad
\Lambda_1=\int_\om f_\te'(1)\;d\mu(\te).
\]
If both $\Lambda_0$ and $\Lambda_1$ are positive, then as the time
goes to $-\infty$, the levels 0 and 1 are attracting. In many cases we
can use Theorem~\ref{basin} to conclude that their basins of
attraction are nontrivial. However, there is no guarantee that the
boundaries of those basins coincide. For this we need some kind of
contraction in the fibers as the time goes to $+\infty$. Since the
fiber maps are homeomorphisms, we cannot get contractions on closed
intervals $[0,1]$. However, sometimes there is a kind of contraction
on the open intervals $(0,1)$. One example of such a situation is
given in the paper~\cite{BoMi}. There all maps $f_\te$ have positive
Schwarzian derivative. Later in our paper we give a completely
different example with two piecewise linear maps. However, there is no
standard method of proving forward contraction for homeomorphisms.
Therefore in our general theorem that follows, we make it one of the
assumptions. In particular, we will use the following terminology,
independently whether $S$ is invertible or not.

\begin{definition}\label{esscontr}
The skew product $F:\om\times I\to\om\times I$ is \emph{essentially
  contracting} if for almost all $\te\in\om$ and all $x,y\in(0,1)$,
the distance
\[
|\pi_2(F^n(\te,x))-\pi_2(F^n(\te,y))|
\]
goes to $0$ as $n\to\infty$.
\end{definition}

If $\psi:\omega\to I$ is a measurable function, then we define the
measure $\mu_\psi$, concentrated on the graph of $\psi$, as the
lifting of the measure $\mu$, that is,
\[
\mu_\psi(A)=\mu\{\te\in\om:(\te,\psi(\te))\in A\}.
\]

\begin{theorem}\label{main}
For a skew product $F$ as above, assume that
\begin{enumerate}[(I)]
\item\label{mi1} $\Lambda_0,\Lambda_1>0$,
\item\label{mi2} either the set $\{f_\te:\te\in\om\}$ is finite, or
all $f_\te$ are diffeomorphisms of class $C^2$ with
$|f_\te''|/(f_\te')^2$ bounded uniformly in $\te$ and $x$,
\item\label{mi3} $F$ is essentially contracting.
\end{enumerate}
Then there exists a measurable function $\phi:\om\to (0,1)$ with the
following properties:
\begin{enumerate}[(a)]
\item\label{ma1} for almost every $\te\in\om$, if $x<\phi(\te)$ then
\begin{equation}\label{co0}
\lim_{n\to\infty}\pi_2(F^{-n}(\te,x))=0
\end{equation}
and if $x>\phi(\te)$ then
\begin{equation}\label{co1}
\lim_{n\to\infty}\pi_2(F^{-n}(\te,x))=1,
\end{equation}
\item\label{ma2} the graph of $\phi$ is $F$-invariant,
\item\label{ma3} for almost every $\te\in\om$ and every $x\in(0,1)$,
\[
\lim_{n\to\infty}|\pi_2(F^n(\te,x))-\phi(S^n(\te))|=0,
\]
\item\label{ma4} for almost every $\te\in\om$ and for every compact
set $A\subset(0,1)$ and $\eps>0$ there exists $N$ such that for every
$n\ge N$
\begin{equation}\label{co3}
\pi_2(F^n(\{S^{-n}(\te)\}\times A)\subset (\phi(\te)-\eps,
\phi(\te)+\eps).
\end{equation}
\item\label{ma5} if $\om$ is a metric compact space and $F$ is
continuous, then for almost every $\te\in\om$ and every $x\in(0,1)$,
the measures
\[
\frac1n\sum_{k=0}^{n-1}F^k_*(\delta_{(\te,x)})
\]
converge (as $n\to\infty$) in the weak-$*$ topology to the measure
$\mu_\phi$.
\end{enumerate}
\end{theorem}

\begin{proof}
Let us start by proving that the assumptions of Theorem~\ref{basin}
are satisfied for $G=F^{-1}$. Clearly, the exponent for $G$ at level 0
is equal to $-\Lambda_0$, so it is negative. Then, if there are
finitely many fiber maps for $F$, then there are finitely many fiber
maps for $G$, so~\eqref{i1} of Theorem~\ref{basin} is satisfied. If all
$f_\te$ are diffeomorphisms of class $C^2$ with $|f_\te''|/(f_\te')^2$
bounded uniformly in $\te$ and $x$, then to show that~\eqref{i3} of
Theorem~\ref{basin} is satisfied, we just use the formula
\[
\frac{(f^{-1})''(x)}{(f^{-1})'(x)}=\frac{-f''(f^{-1}(x))}
{(f'(f^{-1}(x)))^2}.
\]

Thus, by Theorem~\ref{basin}, there exists a measurable function
$\phi:\om\to (0,1]$, such that for almost every $\te\in\om$, if
$x<\phi(\te)$ then~\eqref{co0} holds. Similarly, there exists a
measurable function $\tilde\phi:\om\to [0,1)$, such that for almost
every $\te\in\om$, if $x>\tilde\phi(\te)$ then~\eqref{co1} holds for
$\phi$ replaced by $\tilde\phi$. Clearly, $\phi\le\tilde\phi$, so both
functions have values in $(0,1)$. By Remark~\ref{invgr}, the graphs of
both functions are $F$-invariant (in particular,~\eqref{ma2} holds).
This means that $\phi(S^n(\te))= \pi_2(F^n(\te,\phi(\te))$.
Thus,~\eqref{ma3} follows from~\eqref{mi3}. Similarly,~\eqref{ma3}
holds with $\phi$ replaced by $\tilde\phi$.

In such a way we get that
\begin{equation}\label{co2}
\lim_{n\to\infty}|\phi(S^n(\te))-\tilde\phi(S^n(\te))|=0
\end{equation}
for almost every $\te$. We want to prove that $\phi=\tilde\phi$ almost
everywhere. If this is not true, then there exists $\eps>0$ and a set
$A\subset\om$ of positive measure such that $|\phi(\te)-
\tilde\phi(\te)|>\eps$ for every $\te\in A$. However, by ergodicity of
$\mu$, the trajectory of almost every point of $\om$ passes through
$A$ infinitely many times, so we get a contradiction with~\eqref{co2}.
Thus, $\phi=\tilde\phi$ almost everywhere, and this completes the
proof of~\eqref{ma1}.

To prove~\eqref{ma4}, observe that there is $\delta>0$ such that
$A\subset (\delta,1-\delta)$. Take $\te$ for which~\eqref{ma1} holds.
Then there is $N$ such that if $n\ge N$ then $\pi_2(F^n(\te,
\max(\phi(\te-\eps),0))<\delta$ and $\pi_2(F^n(\te,
\min(\phi(\te+\eps),1))>1-\delta$. Then~\eqref{co3} holds.

To prove~\eqref{ma5}, take $\te$ for which~\eqref{ma3} holds and such
that $(\te,\phi(\te))$ is generic for $\mu_\phi$. The set of such $\te$
has full measure. If $x\in(0,1)$ then the distance between
$F^n(\te,x)$ and $F^n(\te,\phi(\te))$ goes to 0 as $n\to\infty$, and
therefore~\eqref{ma5} holds.
\end{proof}

Let us finish this section by proving a theorem on invariant measures.
It holds whether $S$ (and therefore, $F$) is invertible or not. Its
proof is basically taken from \cite{BMS}. We assume in it that there
is topology in $\om$ in which $\mu$ is a Borel measure.

\begin{theorem}\label{onemeasure}
Assume that $F$ is an essentially contracting skew product as above.
Then there is at most one ergodic probability measure invariant for $F$
that projects to $\mu$ under $(\pi_2)_*$ and such that the measure of
$\om\times\{0,1\}$ is $0$.
\end{theorem}

\begin{proof}
If there are two such measures, say $\nu_1$ and $\nu_2$, there is
$\te\in\om$ and two points $x,y\in(0,1)$, such that $(\te,x)$ is
generic for $\nu_1$, $(\te,y)$ is generic for $\nu_2$, and
\begin{equation}\label{om}
\lim_{n\to\infty}|\pi_2(F^n(\te,x))-\pi_2(F^n(\te,y))|=0.
\end{equation}
Then in the weak-* topology, the averages of the images of the Dirac
delta measure at $(\te,x)$ converge to $\nu_1$ and the averages of the
images of the Dirac delta measure at $(\te,y)$ converge to $\nu_2$,
and by~\eqref{om} we get $\nu_1=\nu_2$.
\end{proof}

\section{Bernoulli shift in the base}\label{bernoulli}

Let us assume now that $(S,\om,\mu)$ is a Bernoulli shift on a finite
alphabet. We can consider a two-sided shift $(\sigma,\Sigma,\mu)$ or a
one-sided shift $(\sigma_+,\Sigma_+,\mu_+)$. We will write the points
of $\Sigma$ and $\Sigma_+$ as $\uom=(\omega_n)_{n=\infty}^\infty$ or
$\uom=(\omega_n)_{n=0}^\infty$ respectively. We will also assume that
the maps $f_{\uom}$ depend only on $\omega_0$ (so there are only
finitely many of them). The interpretation is that we are choosing
those maps randomly and independently each time.

There is a natural projection $P:\Sigma\to\Sigma_+$. It is a
semiconjugacy and it sends the measure $\mu$ to $\mu_+$.

In this context, let us look closer at the definition of the function
$\phi$, given at the beginning of Section~\ref{sec-basin}.

\begin{lemma}\label{pastphi}
If $\uom=(\omega_n)_{n=-\infty}^\infty$, then $\phi(\uom)$ depends
only on $\omega_n$ with $n<0$.
\end{lemma}

\begin{proof}
In our case, we have
\[
\phi_{n,m}(\uom)=\pi_2(F^n(\sigma^{-n},1/m)),
\]
so it depends only on $\omega_n$ with $n<0$. Thus, the same is true
for $\phi(\uom)$.
\end{proof}

Now we can look what what happens when we project the measure
$\mu_\phi$ to the one-sided system.

\begin{theorem}\label{projection}
There exists a probability measure $\nu$ on $(0,1)$ such that
\[
(P\times\id_I)_*(\mu_\phi)=\mu_+\times\nu.
\]
\end{theorem}

\begin{proof}
We can write $\Sigma=\Sigma_-\times\Sigma_+$, with
\[
\uom=(\omega_n)_{n=-\infty}^\infty=(\uom_-,\uom_+)=
((\omega_n)_{n=-\infty}^{-1}, (\omega_n)_{n=0}^\infty),
\]
where $\uom_-\in \Sigma_-$ and $\uom_+\in\Sigma_+$. By
Lemma~\ref{pastphi}, there exists a measurable function
$\phi_-:\Sigma_-\to(0,1)$ such that
\begin{equation}\label{eq-pr}
\phi(\uom)=\phi_-(\uom_-).
\end{equation}
On $\Sigma_-$ there is a product measure $\mu_-$ such that
$\mu=\mu_-\times \mu_+$. We can identify in a natural way
$\Sigma\times I=(\Sigma_-\times\Sigma_+)\times I$ with
$\Sigma_+\times(\Sigma_-\times I)$. Then, by~\eqref{eq-pr}, we have
$\mu_\phi=(\mu_-)_{\phi_-}\times\mu_+$, where $(\mu_-)_{\phi_-}$ is
the measure on $\Sigma_-\times I$ defined similarly as $\mu_\phi$.

Let $\pi_-:\Sigma_-\times I\to I$ be the natural projection. Set
$\nu=(\pi_-)_*((\mu_-)_{\phi_-})$. With our identification, we have
$P\times\id_I=\id_{\Sigma_+}\times\pi_-$. We get
\[
(P\times\id_I)_*(\mu_\phi)=(\id_{\Sigma_+}\times\pi_-)_*
(\mu_+\times(\mu_-)_{\phi_-})=(\id_{\Sigma_+})_*(\mu_+)
\times(\pi_-)_*((\mu_-)_{\phi_-})=\mu_+\times\nu.
\]
\end{proof}

\section{Piecewise linear homeomorphisms}\label{plh}

Now we consider a one-parameter family of random homeomorphisms of an
interval, for which we can prove that the theory from the preceding
sections applies.

The situation will be as in the preceding section. The system in the
base will be the Bernoulli shift with probabilities $(1/2,1/2)$. The
corresponding interval homeomorphisms, $f_0,f_1:I\to I$ will be
piecewise liner with two pieces. Additionally, their graphs will be
symmetric with respect to $(1/2,1/2)$, that is, $f_1(x) = 1 -
f_0(1-x)$. For each map the point at which it is not linear can be
considered as a critical point. As always, the situation is simpler if
there is only one critical value, and by the symmetry, this common
critical value has to be $1/2$. Since our maps are orientation
preserving homeomorphisms, we have $f_0(0) = f_1(0) = 0$ and $f_0(1) =
f_1(1) = 1$.

These conditions determine a one-parameter family of pairs of maps
\begin{align*}
f_0(x) & = \begin{cases}
ax & \text{if $0 \le x \le 1-c$,} \\
1-b(1-x) & \text{if $1-c \le x \le 1$,}
\end{cases} \\
f_1(x) & = \begin{cases}
bx & \text{if $0 \le x \le c$,} \\
1-a(1-x) & \text{if $c \le x \le 1$.}
\end{cases}
\end{align*}
where $a=\frac1{2(1-c)}$, $b=\frac1{2c}$, and $0<c<1/2$ (see
Figure~\ref{mapsf0f1}). Observe that the harmonic mean of the slopes
$a$ and $b$ is 1, and that $0<a<1<b$.
\begin{figure}[ht]
\includegraphics{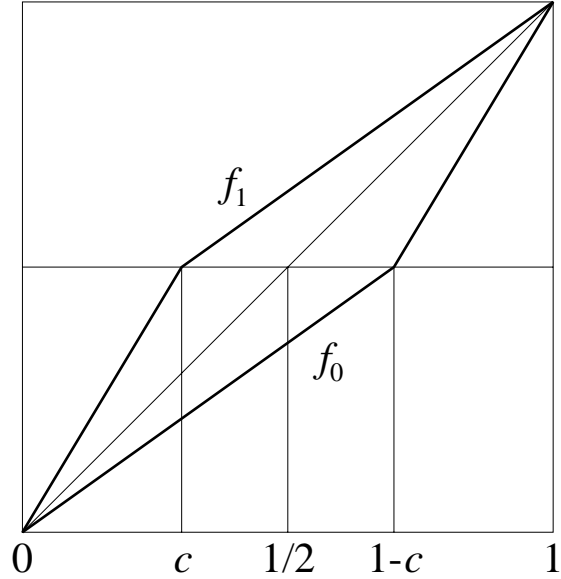}
\caption{The maps $f_0$ and $f_1$}\label{mapsf0f1}
\end{figure}

We will apply $f_j$, $j=0,1$, when the 0-th coordinate of
$\uom\in\Sigma$ (or in $\Sigma_+$) is $j$. That is, we consider skew
products $F:\Sigma\times I\to \Sigma\times I$ given by $F(\uom,x) =
(\sigma(\uom), f_{\omega_0}(x)),$ where $\uom =
(\omega_n)_{n=-\infty}^\infty$, and $F_+:\Sigma_+\times I\to
\Sigma_+\times I$ given by $F_+(\uom,x) = (\sigma_+(\uom),
f_{\omega_0}(x)),$ where $\uom = (\omega_n)_{n=0}^\infty$.

We want to apply Theorem~\ref{main}. Therefore we need to check that
its assumptions are satisfied by $F$. Assumption~\eqref{mi1} is
satisfied because $ab=\frac1{4c(1-c)}>1$. Assumption~\eqref{mi2} is
satisfied because there are only 2 maps $f_\te$. Thus, we have to
prove that $F$ is essentially contracting. As we mentioned earlier,
this is a nontrivial thing to do.

The main idea is to find a homeomorphism from $(0,1)$ to $\R$ such
that in the new metric in $(0,1)$, which we get by transporting back
the natural metric from $\R$, both maps $f_0$ and $f_1$ are
contractions. In fact, they will be very weak contractions (on the
most of the space they will be isometries), so we need more work in
order to prove that $F$ is essentially contracting.

Let $h:(0,1)\to\R$ be a homeomorphism given by the formula
\[
h(x)=\begin{cases}
\log x-\log\tfrac12 & \text{if $x\le\frac12$,}\\
\log\tfrac12-\log(1-x) & \text{if $x>\frac12$.}
\end{cases}
\]
Then we use the metric $d(x,y)=|h(x)-h(y)|$. We can rewrite it as
\[
d(x,y) = \begin{cases}
|\log(x) - \log(y)| & \text{if $x,y \in (0, 1/2]$}\\
|\log(1-x) - \log(1-y)| & \text{if $x,y \in [1/2, 1)$}
\end{cases}
\]
and $d(x,y) = d(x,1/2) + d(y,1/2)$ in any other case. Clearly, $d$ is
a metric in $(0,1),$ equivalent to the Euclidean one.

\begin{remark}\label{equaltriang}
If $x \le y \le z$ then $d(x,z) = d(z,y) + d(y,z)$.
\end{remark}

Now we start the study of the contraction of $F$.

\begin{lemma}\label{deriv}
Assume that $1/2\le x<y<1$. Then
\begin{equation}\label{est1}
\frac{\log y-\log x}{\log(1-x)-\log(1-y)}\le\frac{4-2y}3.
\end{equation}
\end{lemma}

\begin{proof}
We have
\[
\frac{\log y-\log x}{\log(1-x)-\log(1-y)}=
\frac{\frac{\log y-\log x}{y-x}}{\frac{\log(1-x)-\log(1-y)}
{(1-x)-(1-y)}}.
\]
Since the logarithmic function is concave, the numerator of the
right-hand side above is a decreasing function of $x$, while the
denominator is a decreasing function of $1-x$, that is, an increasing
function of $x$. Therefore, the whole fraction is a decreasing
function of $x$. Thus,
\begin{equation}\label{est2}
\frac{\log y-\log x}{\log(1-x)-\log(1-y)}\le
\frac{\log y-\log\frac12}{\log\frac12-\log(1-y)}=
\frac{\log2y}{-\log2(1-y)}.
\end{equation}

Assume that $0\le t<1$. We use two well-known estimates of the
logarithm, namely
\[
\log(1+t)\le t \qquad\text{and}\qquad -\log(1-t)\ge t+\frac{t^2}2.
\]
{}From those inequalities we get
\begin{equation}\label{est3}
\frac{\log(1+t)}{-\log(1-t)}\le\frac{t}{t+\frac{t^2}2}=
\frac2{2+t}.
\end{equation}
We claim that
\begin{equation}\label{est4}
\frac2{2+t}\le\frac{3-t}3.
\end{equation}
Indeed, this is equivalent to $6\le 6-2t+3t-t^2$, that is, to
$t(1-t)\ge 0$, which is true under our assumptions. From~\eqref{est3}
and~\eqref{est4} we get
\[
\frac{\log(1+t)}{-\log(1-t)}\le\frac{3-t}3.
\]
Applying this inequality to $t=2y-1$, we get
\[
\frac{\log2y}{-\log2(1-y)}\le\frac{4-2y}3.
\]
Together with~\eqref{est2}, we obtain~\eqref{est1}.
\end{proof}

\begin{lemma}\label{f-contr}
If either $x,y \in (0,1/2]$ or $x,y \in [1-c,1)$ then
$d\left(f_0(x),f_0(y)\right) = d(x,y)$. If $x,y \in [1/2, 1-c]$ then
\begin{equation}\label{est5}
d\left(f_0(x),f_0(y)\right)\le\left(1-\frac{2c}3 d(x,y)\right) d(x,y).
\end{equation}
If $x,y \in (0, c]$ or $x,y \in [1/2, 1)$ then
$d\left(f_1(x),f_1(y)\right) = d(x,y)$. If $x,y \in [c, 1/2]$
then~\eqref{est5} holds with $f_1$ instead of $f_0$.
\end{lemma}

\begin{proof}
We will only prove the statements for $f_0$.
The statements for $f_1$ follow in a similar way (or one can use
symmetry).

If $x,y \in (0,1/2],$ then
\[
d\left(f_0(x),f_0(y)\right) = \abs{\log(ax)-\log(ay)} =
  \abs{\log(x)-\log(y)} = d(x,y).
\]
When $x,y \in [1-c,1)$ we also obtain
$d\left(f_0(x),f_0(y)\right) = d(x,y)$ in a similar way.

Now assume that
$x,y \in [1/2, 1-c]$ and $x < y.$ Then
\[
d(x,y) = \log(1-x)-\log(1-y)
\qquad\text{and}\qquad
d\left(f_0(x),f_0(y)\right) = \log y-\log x.
\]
Thus, by Lemma~\ref{deriv},
\begin{equation}\label{est6}
d\left(f_0(x),f_0(y)\right) \le \frac{4-2y}3 d(x,y).
\end{equation}

On the interval $[c,1/2]$ the logarithmic function is Lipschitz
continuous with the constant $1/c$. Therefore
\[
cd(x,y)=c\big(\log(1-x)-\log(1-y)\big)\le y-x\le y-\frac12,
\]
so
\[
\frac{4-2y}3=1-\frac{2y-1}3=1-\frac23\left(y-\frac12\right)
\le 1-\frac{2c}3 d(x,y).
\]
{}From this and~\eqref{est6} we get~\eqref{est5}.
\end{proof}

Fix $\uom \in \Sigma$. For $x_0 \in [0,1]$ we will write
$x_n=\pi_2(F^n(\uom,x_0))$. Set
\[
\Gamma=\left\{\uom\in \Sigma:\lim_{n\to\infty}\#\big\{
k\in\{0,1,\dots,n-1\}:\omega_k = 0\big\} = \frac12\right\}.
\]
By the Birkhoff Ergodic Theorem, $\mu(\Gamma) = 1$.

In what follows, given $\uom \in \Gamma$ and $x_0 \in (0,1)$,
for $n \ge 1$ we define $x_n := f_{\omega_{n-1}}(x_{n-1})$.
Observe that
$F^n(\uom, x_0) = F(S^{n-1}(\uom), x_{n-1}) = (S^n(\uom), x_n).$

\begin{lemma}\label{InfVisits}
Let $\uom \in \Gamma$ and $x_0 \in (0,1)$. Then there are infinitely
many values of $n$ such that $x_n \in (0,1/2]$ and infinitely many
values of $n$ such that $x_n \in [1/2, 1).$
\end{lemma}

\begin{proof}
Suppose that there are only finitely many $n$'s such that
$x_n \in [1/2, 1).$
Without loss of generality we may assume that there are no such $n$'s.
Then
$\omega_n = 0$ implies $x_{n+1} = a x_n$ and
$\omega_n = 1$ implies $x_{n+1} = b x_n.$
Take $\eps > 0$ such that
\[
 \eps < \frac{\log(ab)}{2\log\left(\frac{b}{a}\right)}.
\]
Then, $a^{\tfrac{1}{2} + \eps} b^{\tfrac{1}{2} - \eps} > 1$.
Since $\uom \in \Gamma,$ if $n$ is large enough,
\[
 \#\set{k<n}{\omega_n=0} < \left( \frac{1}{2} + \eps \right)n.
\]
Consequently,
\[
 x_n \ge
  a^{\left(\tfrac{1}{2} + \eps\right)n}
  b^{\left(\tfrac{1}{2} - \eps\right)n}
  x_0 = \left(
     a^{\tfrac{1}{2} + \eps}
     b^{\tfrac{1}{2} - \eps}
  \right)^n x_0
\]
and this last expression tends to $\infty$ as $n$ tends to $\infty$;
a contradiction.
\end{proof}

\begin{lemma}\label{d-contr}
For every $x,y \in (0,1)$ we have
\[
d(f_0(x),f_0(y)) \le d(x,y) \qquad\text{and}\qquad
d(f_1(x),f_1(y)) \le d(x,y).
\]
\end{lemma}

\begin{proof}
If both $x,y$ are in one of the intervals $(0,1/2]$ or $[1/2, 1-c],$
or $[1-c, 1),$ then by Lemma~\ref{f-contr} $d(f_0(x),f_0(y)) \le
d(x,y)$. Otherwise, we divide the interval between $x$ and $y$ into
two or three subintervals as above and use Remark~\ref{equaltriang}.

For $f_1$ the proof is similar.
\end{proof}

\begin{lemma}\label{nohalf}
There exists $\eta>0$ such that if $x\le 1/2\le y$ and $d(x,y)<\eta$
then $f_0(x)<f_0(y)<1/2$ and $1/2<f_1(x)<f_1(y)$.
\end{lemma}

\begin{proof}
This follows immediately from the inequality $f_0(1/2)<1/2<f_1(1/2)$
and continuity of $f_0$ and $f_1$.
\end{proof}

\begin{lemma}\label{d-a-n-contr}
Let $1/2\le x_0 < y_0$ and $x_n < y_n\le 1/2$ for some $n \ge 1.$
Assume also that $d(x_0,y_0)<\eta$, where $\eta$ is the constant from
the preceding lemma. Then
\begin{equation}\label{contr1}
d(x_n,y_n) \le \frac{2+\frac{c}3d(x_0,y_0)}
{2+\frac{2c}3d(x_0,y_0)}d(x_0, y_0).
\end{equation}
\end{lemma}

\begin{proof}
Let $k$ be the largest integer from $\{0,1,\dots,n-1\}$ such that
$1/2\le x_k < y_k$. By Lemmas~\ref{nohalf} and~\ref{d-contr}, either
$x_{k+1} = f_0(x_k) < y_{k+1} = f_0(x_k) \le 1/2$ or
$x_{k+2} = f_0(x_{k+1}) < y_{k+2} = f_0(x_{k+1}) \le 1/2$
(in the latter case, $k \le n-2$).

In the first case, by Lemma~\ref{f-contr},
\[
d(x_{k+1},y_{k+1})\le\left(1-\frac{2c}3d(x_k,y_k)\right)d(x_k,y_k),
\]
so by Lemma~\ref{d-contr},
\[
d(x_n,y_n)\le\left(1-\frac{2c}3d(x_n,y_n)\right)d(x_0,y_0).
\]
This inequality implies
\begin{equation}\label{contr2}
d(x_n,y_n)\le\frac1{1+\frac{2c}3d(x_0,y_0)}d(x_0,y_0).
\end{equation}
If $\alpha>0$ then $1/(1+2\alpha)<(2+\alpha)/(2+2\alpha)$,
so~\eqref{contr1} follows in this case.

In the second case there is a point $z_0\in(x_0,y_0)$ such that
$z_{k+1}=1/2$. Then the first case applies if we replace $y_0$ by
$z_0$, and also if we replace $x_0$ by $z_0$. Suppose that
$d(x_0,z_0)\ge d(z_0,y_0)$ (if $d(x_0,z_0)<d(z_0,y_0)$ then the proof
is similar). Then, by~\eqref{contr2} (applied to $x_0$ and $z_0$),
Lemma~\ref{d-contr} and Remark~\ref{equaltriang}, we get
\begin{equation}\label{contr3}
d(x_n,y_n)=d(x_n,z_n)+d(z_n,y_n)\le
\frac1{1+\frac{2c}3d(x_0,z_0)}d(x_0,z_0)+d(z_0,y_0).
\end{equation}
Since $d(x_0,z_0)\ge d(z_0,y_0)$ and
$d(x_0,z_0)+d(z_0,y_0)=d(x_0,y_0)$, we have $d(x_0,z_0)\ge
d(x_0,y_0)/2$, so we can write
\[
d(x_0,z_0)=d(x_0,y_0)/2+\big(d(x_0,z_0)-d(x_0,y_0)/2\big)
\]
with $d(x_0,z_0)-d(x_0,y_0)/2\ge 0$. Thus,
\[
\frac1{1+\frac{2c}3d(x_0,z_0)}d(x_0,z_0)\le
\frac1{1+\frac{2c}3d(x_0,z_0)}\cdot\frac{d(x_0,y_0)}2+
\left(d(x_0,z_0)-\frac{d(x_0,y_0)}2\right).
\]
Together with~\eqref{contr3}, taking into account that
$d(x_0,z_0)+d(z_0,y_0)=d(x_0,y_0)$, we get
\[
d(x_n,y_n)\le\left(\frac1{1+\frac{2c}3d(x_0,z_0)}+1\right)
\frac{d(x_0,y_0)}2.
\]
Using $d(x_0,z_0)\ge d(x_0,y_0)/2$ again, we get
\begin{equation}\label{contr4}
d(x_n,y_n)\le\left(\frac1{1+\frac{c}3d(x_0,y_0)}+1\right)
\frac{d(x_0,y_0)}2=\frac{2+\frac{c}3d(x_0,y_0)}
{2+\frac{2c}3d(x_0,y_0)}d(x_0,y_0).
\end{equation}
Thus,~\eqref{contr1} also follows in this case.
\end{proof}

Define a function $\chi:[0,\infty)\to\R$ by
\[
\chi(t)=\begin{cases}
\dfrac{2+\frac{c}3 t}{2+\frac{2c}3 t}t & \text{if $0\le
  t\le\frac\eta2$,}\\ \\
\dfrac{2+\frac{c\eta}6}{2+\frac{c\eta}3}t & \text{if $t>\frac\eta2$,}
\end{cases}
\]
where $\eta$ is the constant from Lemma~\ref{nohalf}. It is easy to
see that $\chi$ is continuous, $\chi(0)=0$ and $\chi(t)<t$ if $t>0$.
Therefore, for every $t\ge 0$ we have
\begin{equation}\label{limzero}
\lim_{n\to\infty}\chi^n(t)=0.
\end{equation}
It is clear that $\chi$ is strictly increasing on $[\eta/2,\infty]$.
By differentiating the first formula defining $\chi$, one can easily
check that the same is true on $[0,\eta/2]$. Thus, $\chi$ is
invertible and for every $t>0$ we have
\begin{equation}\label{liminfty}
\lim_{n\to\infty}\chi^{-n}(t)=\infty.
\end{equation}

\begin{lemma}\label{long-contr}
Let $1/2\le x_0 < y_0$ and $x_n < y_n\le 1/2$ for some $n \ge 1.$
Then
\begin{equation}\label{l-contr}
d(x_n,y_n) \le \chi(d(x_0, y_0)).
\end{equation}
\end{lemma}

\begin{proof}
If $d(x_0,y_0)\le\eta/2$, then~\eqref{l-contr} follows immediately
from Lemma~\ref{d-a-n-contr} and the definition of $\chi$. If
$d(x_0,y_0)>\eta/2$, then we can
divide the interval $[x_0,y_0]$
by taking points
$x_0=x_0^0 < x_0^1 < x_0^2 < \dots x_0^m = y_0$
such that $\eta/2 \le d(x_0^i,x_0^{i+1}) < \eta$ for $i=0,1,\dots,m-1$,
and apply Lemma~\ref{d-a-n-contr} to each of the intervals
$[x_0^i,x_0^{i+1}]$. We get
\[
 d(x_n,y_n) = \sum_{i=0}^{m-1} d(x_n^i,x_n^{i+1}) \le
 \sum_{i=0}^{m-1}
 \frac{2+\frac{c}3d(x_0^i,x_0^{i+1})}
     {2+\frac{2c}3d(x_0^i,x_0^{i+1})}
 d(x_0^i,x_0^{i+1}).
\]
Since for $t\ge\eta/2$ we have
\[
\frac{2+\frac{c}3 t}{2+\frac{2c}3 t}\le
\frac{2+\frac{c\eta}6}{2+\frac{c\eta}3},
\]
we obtain
\[
 d(x_n,y_n) \le
 \frac{2+\frac{c\eta}6}{2+\frac{c\eta}3}
 \sum_{i=0}^{m-1} d(x_0^i,x_0^{i+1}) =
 \frac{2+\frac{c\eta}6}{2+\frac{c\eta}3} d(x_0,y_0) =
 \chi(d(x_0,y_0)).
\]
\end{proof}

\begin{lemma}\label{forw-contr}
Let $\uom \in \Gamma$ and let $x_0, y_0 \in (0,1)$. Then
$\lim_{n\to\infty} d(x_n,y_n) = 0$.
\end{lemma}

\begin{proof}
We may assume that $x_0 < y_0$. By Lemma~\ref{InfVisits}, there are
increasing sequences $(n_k)$ and $(m_k)$ such that $n_k < m_k <
n_{k+1}$ and $y_{m_k} \le 1/2 \le x_{n_k}.$ By
Lemma~\ref{long-contr} we have $d(x_{m_k},y_{m_k}) \le
\chi(d(x_{n_k}, y_{n_k})).$ By this and Lemma~\ref{d-contr} used
inductively, we get $d(x_{n_{k+1}},y_{n_{k+1}}) \le \chi(d(x_{n_k},
y_{n_k})).$ Thus, by induction, $d(x_{n_k},y_{n_k}) \le
\chi^{k-1}(d(x_{n_1}, y_{n_1})).$ By~\eqref{limzero}, we get $\lim_{n_k
\to \infty} d(x_{n_k},y_{n_k}) = 0.$ Using again Lemma~\ref{d-contr}
inductively and taking into account that $n_k < n_{k+1}$ (so $n_k \to
\infty$ as $k \to \infty$), we get $\lim_{n \to \infty} d(x_{n},y_{n})
= 0.$
\end{proof}

The derivative of the function $h$, which is used to define distance
$d$, is larger than 1. Therefore $|x-y|\le d(x,y)$ for all
$x,y\in(0,1)$. In such a way we get from Lemma~\ref{forw-contr} the
desired result.

\begin{theorem}\label{forw-contr1}
For almost all $\uom\in\Sigma$, if $x_0, y_0 \in (0,1)$ then
$\lim_{n\to\infty} |x_n-y_n| = 0$.
\end{theorem}

\begin{corollary}\label{exmain}
The map $F$ considered in this section satisfies the assumptions of
Theorem~\ref{main}.
\end{corollary}

\begin{remark}\label{1-sid-contr}
If instead of $F$ we consider the map $F_+$, which is a skew product
over the one-sided shift, for a given $\uom\in\Sigma_+$ and $x_0\in I$
we get the same $x_n$ as for $F$ when we replace $\uom$ by any
two-sided sequence with the same $\omega_k$ for $k\ge 0$. Therefore
Theorem~\ref{forw-contr1} holds also if we replace $\Sigma$ by
$\Sigma_+$ and $F$ by $F_+$.
\end{remark}

\section{Measures}\label{sec-measures}

We continue to investigate $F$ and $F_+$, this time from the point of
view of invariant measures. The relevant invariant measures for $F$
and $F_+$ are those that project to $\mu$ and $\mu_+$. There are two
trivial ergodic ones: $\mu\times\delta_0$ and $\mu\times\delta_1$ (in
the one-sided case, $\mu_+\times\delta_0$ and $\mu_+\times\delta_1$).

By Theorem~\ref{onemeasure} and Corollary~\ref{exmain}, there is at most one nontrivial measure
of this type. Such measure for $F$ is $\mu_\phi$, which appears in
Theorem~\ref{main}~\eqref{ma5}. It is clear that the projection from
$\Sigma\times I$ to the first coordinate is an isomorphism of the
systems $(\Sigma\times I,F,\mu_\phi)$ and $(\Sigma,\sigma,\mu)$. In
particular, this shows that $\mu_\phi$ is ergodic for $F$.

Now we consider $F_+$. Here the situation is completely different.
Denote the Lebesgue measure on $I$ by $\lambda$. The following theorem
can be interpreted as the Lebesgue measure being invariant for our
random system of maps. The proof is straightforward and specific for
our family.

\begin{theorem}\label{invmeas}
The measure $\mu_+ \times \lambda$ is invariant for $F_+$.
\end{theorem}

\begin{proof}
Let $\eps_i \in \{0,1\}$ for $i=0,1,\dots,n-1$ and let
\[
C=C(\eps_0,\eps_1,\dots,\eps_{n-1}):=\set{(\omega_0,\omega_1,\dots)}
{\omega_i = \eps_i \text{ for $i=0,1,\dots,n-1$}}
\]
be an $n$-cylinder of the one-sided shift and let $A \subset [0,1/2]$
or $A \subset [1/2,1]$ be a $\lambda$-measurable set.

Then, $F_+^{-1}(C \times A) = (C_0 \times A_0) \cup (C_0 \times A_1),$
where, for $j\in \{0,1\}$, $A_j = f^{-1}_j(A)$ and
\[
C_j = \set{\uom}{\text{$\omega_0 = j$ and $\omega_i = \eps_i$ for
$i=1,2,\dots,n-1$}}.
\]
Since $\tfrac{1}{a} + \tfrac{1}{b} = 2$ we have $\lambda(A_0) +
\lambda(A_1) = \tfrac{\lambda(A)}{a} + \tfrac{\lambda(A)}{b} =
2\lambda(A)$ and clearly $\mu_{+}(C_j) = \tfrac{1}{2} \mu_{+}(C).$
Therefore,
\[
(\mu_{+} \times \lambda)\left(F^{-1}_{+}(C \times A)\right) =
\frac{1}{2} \mu_{+}(C) \cdot 2\lambda(A) = (\mu_{+} \times \lambda)(C
\times A).
\]

The sets of the form $C \times A$ with $C,\ A$ as above generate the
whole $\sigma$-field of $\mu_{+} \times \lambda$-measurable sets. This
completes the proof.
\end{proof}

Once we know this measure, let us compute the Lyapunov exponent in the
direction of the fiber. For each $f_j$, the derivative is $a$ on an
interval of length $1/(2a)$ and $b$ on an interval of length $1/(2b)$.
Therefore the exponent is
\[
\frac1{2a} \log a+ \frac1{2b} \log b .
\]
We have $1/(2a)=1-c$ and $1/(2b)=c$. Therefore
\begin{align*}
\frac1{2a} \log a+ \frac1{2b} \log b& = (1-c)(-\log2-\log(1-c))
+c(-\log2-\log c)\\
& =-(1-c)\log(1-c)-c\log c-\log2.
\end{align*}
Since $0<c<1/2$, this exponent is negative. This agrees with
Theorem~\ref{forw-contr1}.

Recall that $P:\Sigma\to\Sigma_+$ is the natural projection (that
forgets about $\omega_n$ with negative $n$).

\begin{proposition}\label{image_meas}
We have
\[
(P\times\id_I)_*(\mu_\phi)=\mu_+\times\lambda.
\]
\end{proposition}

\begin{proof} We have $P_*(\mu)=\mu_+$, so
$(P\times\id_I)_*(\mu_\phi)$ is a measure invariant for $F_+$. This
measure vanishes on the set $\Sigma_+\times\{0,1\}$, so by
Theorems~\ref{onemeasure} and~\ref{invmeas} it is equal to
$\mu_+\times\lambda$.
\end{proof}

Let us comment on invariant measures for the random systems we are
considering. We assume that $F_+$ is essentially contracting and the
base system is Bernoulli. By Theorem~\ref{onemeasure}, there is one
nontrivial measure invariant for $F_+$ that projects to $\mu_+$. By
Theorem~\ref{projection}, it is of the form $\mu_+\times\nu$ for some
measure $\nu$ on the interval. Thus, the question about the existence
of an absolutely continuous measure for our system is the question
whether this specific measure $\nu$ is absolutely continuous. This is
very different from the situation for non-random interval maps, when
there is a lot of invariant measures and we are asking only whether
there is one among them which is absolutely continuous. We conjecture
that typically (whatever this means) the measure $\nu$ is not
absolutely continuous. The systems considered in Theorem~\ref{invmeas}
are very special, and $\nu=\lambda$ just follows from the definition
of the maps.

Now we can prove some interesting properties of the function $\phi$.

\begin{theorem}\label{dense}
For almost every $x\in I$ the preimage $\phi^{-1}(x)$ is dense in
$\Sigma$. In particular, the graph of $\phi$ is dense in $\Sigma\times
I$.
\end{theorem}

\begin{proof}
Choose a cylinder $C=C(\eps_{-n},\eps_{-n+1},\dots,\eps_n)$, analogous
as in the proof of Theorem~\ref{invmeas}. By Lemma~\ref{pastphi}, on
$\sigma^n(C)$ the function $\phi$ takes all values that it takes on
the whole space. However, by Proposition~\ref{image_meas} and since
$\mu_\phi$ is concentrated on the graph of $\phi$, it takes almost all
values from $I$. By Remark~\ref{invgr},
\[
\phi(\sigma^n(\uom))=\pi_2(F^n(\uom,\phi(\uom))).
\]
If $\uom\in\C$, then
\[
\pi_2(F^n(\uom,\phi(\uom)))=(f_{\eps_{-n+1}}\circ f_{\eps_{-n+2}}
\circ\cdots\circ f_{\eps_0})(\phi(\uom)).
\]
The map $f_{\eps_{-n+1}}\circ f_{\eps_{-n+2}}\circ\cdots\circ
f_{\eps_0}$ is a homeomorphism preserving the Lebesgue equivalence
class, and therefore $\phi$ takes on $C$ almost all values from $I$.

Cylinders form a countable basis of the topological space $\Sigma$ and
the intersection of a countable family of sets of full measure has
full measure. Therefore for almost every $x\in I$ the preimage
$\phi^{-1}(x)$ is dense in $\Sigma$.

The second statement of the theorem follows immediately from the first
one.
\end{proof}

\section{Two-sided vs.\ one sided case}\label{2s1s}

By Theorem~\ref{main}~\eqref{ma3} and Corollary~\ref{exmain}, the map
$F$ has a fiberwise attractor which is a graph of a measurable
invariant function from the base to the fiber space. We will show
that this is not the case if we consider $F_+$, even if we skip the
assumption of invariance.

\begin{theorem}\label{noattr}
There is no measurable function $\phi_+:\Sigma_+ \to (0,1)$ whose
graph is an attractor for $F_+$ in the sense that for almost every
$\uom\in\Sigma_+$ and every $x_0\in(0,1)$ we have
\[
\lim_{n\to\infty}|x_n-\phi_+(\sigma_+^n(\uom))|=0.
\]
\end{theorem}

\begin{proof}
Assume that such $\phi_+$ exists. Then the graph of $\phi_+\circ
P:\Sigma\to(0,1)$ is an attractor for $F$, because $x_n$ depends only
on $x_0$ and on $\omega_k$ with nonnegative $k$. By a theorem
from~\cite{AM}, $\phi_+\circ P=\phi$ almost everywhere. Thus, the
graph of $\phi_+\circ P$ is $F$-invariant, and it follows that the
graph of $\phi_+$ is $F_+$-invariant.

The measure $(\mu_+)_{\phi_+}$ is then a nontrivial $F_+$-invariant
ergodic measure, so by Theorems~\ref{onemeasure} and~\ref{invmeas} it
is equal to $\mu_+ \times \lambda$, a contradiction.
\end{proof}

In such a way we get an excellent illustration of the \emph{Mystery of
  the Vanishing Attractor}, described in~\cite{AM}. For an invertible
system an attractor exists, but it vanishes when we pass to the
noninvertible system. This happens in spite of the fact that in the
definition of an attractor we only look at forward orbits, and that in
the base the future is completely independent of the past.

One can try to explain this paradox by saying that for $F_+$ also
there is an attractor, but it is the whole space. This is true, but
normally when thinking of an attractor one considers subsets much
smaller than the whole space. Another explanation is that when trying
to find an attractor for $F_+$, which is a graph, we try to specify
one point in $(0,1)$ for each $\uom\in\Sigma_+$, without specifying
$x_0$. However, when we know the past, we basically know $x_0$, and
with the knowledge of $x_0$ and $\uom\in\Sigma_+$ we know $x_n$ for
all $n\ge 0$. Again, this is a kind of explanation (due to M.~ Rams),
but still the question why in order to have a nice description of the
future we need the past, if the past and the future are independent,
remains a little mysterious.


\begin{thebibliography}{3}

\bibitem{AM}
Ll.~Alsed{\`a} and M.~Misiurewicz, \textit{Skew Product Attractors and
  concavity}, preprint, 2012.

\bibitem{BMS}
V.~Bergelson, M.~Misiurewicz and S.~Senti, \textit{Affine actions of a
  free semigroup on the real line}, Ergod. Th. Dynam. Sys. \textbf{26}
(2006), 1285-1305.

\bibitem{BoMi}
A.~Bonifant and J.~ Milnor, \textit{Schwarzian derivatives and
  cylinder maps}, in ``Holomorphic dynamics and renormalization'',
Fields Inst. Commun. \textbf{53}, Amer. Math. Soc., Providence, RI,
2008, pp. 1-21.

\bibitem{devan}
W.~de~Melo and S.~van~Strien, ``One-Dimensional Dynamics'', Springer
Verlag, Berlin, 1983.

\bibitem{Mis}
M.~Misiurewicz, \textit{Absolutely continuous measures for certain
  maps of an interval}, Publ. Math. IHES \textbf{53} (1981), 17-51.

\end{thebibliography}
\end{document}